\documentclass[11pt]{amsart}

\usepackage{graphicx, pinlabel}
\newtheorem{theorem}{Theorem}[section]
\newtheorem{proposition}[theorem]{Proposition}

\newtheorem{corollary}[theorem]{Corollary}

\theoremstyle{definition}
\newtheorem{definition}[theorem]{Definition}
\newtheorem{notation}[theorem]{Notation}
\newtheorem{example}[theorem]{Example}
\newtheorem{remark}[theorem]{Remark}

\numberwithin{equation}{section}

\newcommand{\fraction}[3][3]{\mbox{$#2$\makebox[0pt][l]{\mbox{$\mathrm{/}$}}\hspace*{#1 pt}\mbox{$\mathrm{/}$}\hspace{0.2pt}$#3$}}

\newcommand{\Mod}{\operatorname{Mod}}

\renewcommand{\O}{\operatorname{{\mathcal O}}}

\author{Kashyap Rajeevsarathy}
\address{Department of Mathematics\\
Indian Institute of Science Education Research Bhopal\\
ITI (Gas Rahat) Building, Govindpura\\
Bhopal - 462023, Madhya Pradesh\\
India}
\urladdr{home.iiserb.ac.in/$_{\widetilde{\phantom{n}}}$kashyap/}
\email{kashyap@iiserb.ac.in}
\date{\today}
\keywords{surface, mapping class, Dehn twist, nonseparating curve, root, fractional power}

\begin{document}

\title[Fractional powers of Dehn twists about nonseparating curves]
{Fractional powers of Dehn twists about nonseparating curves}

\begin{abstract}
Let $S_g$ be a closed orientable surface of genus $g \geq 2$ and $C$ a simple closed nonseparating curve in $F$. Let $t_C$ denote a left handed Dehn twist about $C$. A \textit{fractional power} of $t_C$ of \textit{exponent} $\fraction{\ell}{n}$ is an $h \in \Mod(S_g)$ such that $h^n = t_C^{\ell}$. Unlike a root of a $t_C$, a fractional power $h$ can exchange the sides of $C$. We derive necessary and sufficient conditions for the existence of both side-exchanging and side-preserving fractional powers. We show in the side-preserving case that if $\gcd(\ell,n) = 1$, then $h$ will be isotopic to the $\ell^{th}$ power of an $n^{th}$ root of $t_C$ and that $n \leq 2g+1$. In general, we show that $n \leq 4g$, and that side-preserving fractional powers of exponents $\fraction{2g}{2g+2}$ and $\fraction{2g}{4g}$ always exist. For a side-exchanging fractional power of exponent $\fraction{\ell}{2n}$, we show that $2n \geq 2g+2$, and that side-exchanging fractional powers of exponent $\fraction{2g+2}{4g+2}$ and $\fraction{
4g+1}{4g+2}$ always exist. We give a complete listing of certain side-preserving and side-exchanging fractional powers on $S_5$.
\end{abstract}

\maketitle

\section{Introduction}
\label{sec:intro}

Let $S_g$ be a closed orientable surface of genus $g \geq 2$ and $C$ be a simple closed nonseparating curve in $S_g$. Let $t_C$ denote a left handed Dehn twist about $C$ and let $\Mod(S_g)$ denote the mapping class group of $S_g$.

A \textit{root of $t_C$ of degree $n$} is an $h \in \Mod(S_g)$ such that $h^n = t_C$. In 2008, D. Margalit and S. Schleimer \cite{MS} showed the existence of
degree $2g+1$ roots of a Dehn twist $t_C$ on $S_{g+1}$ (for $g \geq 1$). In an earlier collaborative work with D. McCullough~\cite{MK1}, we derived necessary and sufficient
conditions for the existence of a root of degree $n$. The geometric construction of a root of degree $n$ of $t_C$ on $S_g$ started with the definition of $C_n$-action on $S_g$ with fixed points $P$ and $Q$ so that the rotation angles induced by the action around these points differ by $2\pi/n$. We then remove invariant  disks around $P$ and $Q$ and attach an annulus $N$, extending the restricted homeomorphism over $N$ using a homeomorphism whose $n^{th}$ power is a full twist
of $N$. Using Thurston's orbifold theory~\cite{T1} (see also~\cite{S1}) and some elementary number theory an equivalent algebraic theory of roots was developed that completely captured this geometric construction. A natural question is whether this theory can be extended to $n^{th}$ roots of  $\ell^{th}$ powers of $t_C$ and whether such roots could possess some additional properties. We will call such a root $h$ a \textit{fractional power} of $t_C$ of \textit{exponent} $\fraction{\ell}{n}$.

\begin{definition}
A \textit{fractional power} of $t_C$ of \textit{exponent} $\fraction{\ell}{n}$ is an $h \in \Mod(S_g)$ such that $h^n = t_C^{\ell}$.
\end{definition}

In particular, a root of $t_C$ of degree $n$ is just a fractional power of exponent $\fraction{1}{n}$. In this paper, we will describe the geometric construction of a fractional power of $t_C$ of exponent $\fraction{\ell}{n}$.  While this construction is fairly straightforward, the main mathematics of the paper is in the extension of the algebraic theory of roots to the case of fractional powers so that it describes their geometric construction. This algebra along with a simple calculus enables us to obtain several qualitative and quantitative results on fractional powers, including their enumeration. The use of the notation $\fraction{\ell}{n}$ instead of $\ell/n$ is for the reason that fractional powers of exponent $\fraction{\ell}{n}$, where $\ell\mid n$ can exist, while powers of exponent $\fraction{1}{(n/\ell)}$ do not. For example, there always exists a fractional power of $t_C$ of exponent $\fraction{2g}{4g}$ in $\Mod(S_{g+1})$ (see Remark~\ref{rem:ub-SP}), but we know from ~\cite{MK1} that a square root of $t_C$ cannot exist.  

Let $h$ be a fractional power of $t_C$ of exponent $\fraction{\ell}{n}$. As in the case of a root of $t_C$, $h$ would also preserve $C$,
which is apparent from the following argument. Since $t_C^{\ell}= ht_C^\ell h^{-1}=t_{h(C)}^\ell$, $h(C)$ is isotopic to $C$, and by isotopy,
we may assume that $h(C)=C$.  We showed in \cite{MK1} that no root of $t_C$ can exchange the two sides of $C$. However, an intriguing fact
about fractional powers of $t_C$ is that they can exchange the sides of $C$, which motivates the following definition.

\begin{definition}
A fractional power is \textit{side-exchanging} if it interchanges the two sides of $C$, and \textit{side-preserving}, otherwise.
\end{definition}

Since $h$ is a root of degree $n$ of $t_C^\ell$, $t_C^rh$ is a root of degree $n$ of $t_C^{\ell + rn}$. We may assume that $\ell \neq n$.
For if $h^n = t_C^n$ then $h$ fixes $C$ up to isotopy and commutes with $t_C$. So ${(ht_C^{-1})}^n = 1$ and $h = k t_C$ for some finite order
homeomorphism $k$ with $k(C)=C$. In other words, $h$ is a trivial modification of a $\fraction{0}{n}$-root that preserves $C$. Consequently, we need only to understand the fractional powers of
$t_C$ having $1 \leq \ell < n$, and we will generally assume that $\ell$ lies in this range.

The main result in both the side-preserving and side-exchanging cases will be proved using Thurston's orbifold theory. We know from~\cite{MK1} that any root of $t_C$ is side-preserving. So the theory of roots derived in~\cite{MK1} naturally extends to the case of side-preserving fractional powers. As in the geometric construction of roots, we define a $C_n$-action on $S_g$ that has two distinguished fixed points $P$ and $Q$. However, the only difference is
that the rotation angles at $P$ and $Q$ have to differ by $2\pi\ell/n$, and the twisting on annulus $N$ is through an angle $2\pi \ell/n$ rather than $2\pi/n$. The quotient orbifold of the $C_n$-action has two distinguished cone points of order $n$. In Section~\ref{sec:SP}, we define an abstract tuple called an \textit{SP data set}, which is an extension of the data set in~\cite{MK1}. An SP data set, in addition to holding the essential algebraic information required to describe the quotient orbifold action, also holds information required that determines the geometric construction a root. The main theorem in Section~\ref{sec:SP}
(Theorem~\ref{thm:SP-main}) asserts that conjugacy classes of side-preserving fractional powers correspond to SP data sets. An interesting consequence of this theorem is the following proposition.
{
\renewcommand{\thetheorem}{\ref{coro:SP-powerofroot}}
\begin{proposition}
Let $h$ be a side-preserving fractional power of $t_C$ of exponent $\fraction{\ell}{n}$ such that $\gcd(\ell,n)=1$. Then $h = {(h')}^\ell$ for some root $h'$ of~$t_C$ of degree $n$.
\end{proposition}
}
\noindent In other words, if $\gcd(\ell,n) = 1$, then a fractional power is essentially the $\ell^{th}$ power of a root degree $n$.
Among other direct applications of Theorem~\ref{thm:SP-main}, is the following corollary.
{
\renewcommand{\thetheorem}{\ref{coro:SP-relupperbound}}
\begin{corollary}
Suppose that $h$ is a side-preserving fractional power of $t_C$ of exponent $\fraction{\ell}{n}$. Then
\begin{itemize}
\item[(a)] $n$ is odd if $\ell$ is odd.
\item[(b)] $n \leq 2g+1$ if $\gcd(\ell,n) = 1$.
\end{itemize}
\end{corollary}
}
\noindent Corollary~\ref{coro:SP-relupperbound} gives an upper bound for $n$ when $\ell$ and $n$ are relatively prime. In the following corollary, we also derive a general upper and lower bound for $n$.
{
\renewcommand{\thetheorem}{\ref{coro:SP-absupperbound}}
\begin{corollary}
Suppose that $h$ is a side-preserving fractional power of $t_C$ of exponent $\fraction{\ell}{n}$ whose conjugacy class is given by the SP data set
$D = ((\ell,n),g_0,(a,b);(c_1,n_1),\dots,(c_m,n_m))$. Then \[\frac{2g+m}{2g_0+m} \leq n \leq \frac{4g}{4g_0+m}\ .\]
\end{corollary}
}
\noindent Finally, we give a complete classification in $\Mod(S_5)$ (up to conjugacy) of side-preserving fractional powers that arise from cyclic actions whose quotient orbifold is topologically a sphere with three cone points. We shall define such fractional powers as \textit{essential fractional powers}.

A side-exchanging fractional power $h$ of $t_C$ will have an exponent of the form $\fraction{\ell}{2n}$ as it is obtained from a $C_{2n}$ action on $S_g$ that has two distinguished fixed points $P$ and $Q$ interchanged by a generator $h'$ of $C_{2n}$. Since the actions at $P$ and $Q$ are conjugate by $h'$, $P$ and $Q$ will have the same local turning angle and will descend to a single cone point of order $n$ in the quotient orbifold. As in the side-preserving case, we define an \textit{SE data set} to encode the algebraic information relating to the geometric construction of a side-exchanging fractional power.  The main theorem in Section~\ref{sec:SE} (Theorem~\ref{thm:SE-main}) establishes that SE data sets correspond to conjugacy classes of side-exchanging fractional powers. Since we know from~\cite{MK1} that side-exchanging (or even degree) roots do not exist, side-exchanging fractional powers cannot be powers of roots. But it is a natural question to ask whether there exist side-exchanging fractional powers that are powers of other (side-exchanging) fractional powers.  It is immediately apparent that a side-exchanging fractional power of exponent $\fraction{\ell}{2n}$, where $\ell$ is prime, can never be such a fractional power. But, when $\ell$ is composite, such a fractional can exist if it satisfies the condition given in the following proposition. (This proposition can be viewed as an analog of Proposition~\ref{coro:SP-powerofroot} for the side-exchanging case).
{
\renewcommand{\thetheorem}{\ref{coro:SE-powerofroot}}
\begin{proposition}
Let $h$ be a side-exchanging fractional power of $t_C$ of exponent $\fraction{\ell}{2n}$ such that $\ell$ is composite integer with $\gcd(\ell,n)=1$. Let $r$ be a divisor of $\ell$. Then $h = {(h')}^r$ for some side-exchanging fractional power $h'$ of~$t_C$ of exponent $\fraction{\ell'}{2n}$.
\end{proposition}
}
From a result of Wiman~\cite{W1} (and later Harvey~\cite{WJH}), we know that $2n\leq 4g+2$, and in Remark~\ref{rem:SE-dataset}, we provide an SE data set that represents the conjugacy class of a side-exchanging fractional power of exponent $\fraction{4g+1}{4g+2}$ in $\Mod(S_{g+1})$, for all $g \geq 1$. The existence of  fractional powers of exponent $\fraction{4g+1}{4g+2}$ makes upper bounds for $2n$ rather superfluous. However, in the following corollary we derive a lower bound for $2n$.
{
\begin{corollary}
\renewcommand{\thetheorem}{\ref{coro:SE-ellandn}}
Suppose that $h$ is a SE fractional power of $t_C$ of exponent $\fraction{\ell}{2n}$ whose conjugacy class is given by the SE data set $D = ((\ell,2n),g_0,a;(k_1,n_1),\ldots,(k_m,n_m))$. Then
\begin{itemize}
\item[(a)] $n$ is odd if $\ell$ is odd, and
\item[(b)] $2n \geq \frac{2g+m}{2g_0+m-1}$
\end{itemize}
\end{corollary}
}
\noindent We also show that a side-exchanging fractional power of exponent $\fraction{2}{2g+2}$ always exists. As in the side-preserving case, here too we give a complete classification of essential fractional powers in $\Mod(S_5)$.

Though, one would intuitively expect the occurrence of side-exchanging fractional powers to be more restrictive, the data obtained using GAP software~\cite{SPSOFTWARE, SESOFTWARE} seems to suggest that, in general,  side-exchanging fractional powers achieve more exponents than side-preserving powers. However, side-exchanging fractional powers (in general) are much fewer in number when compared with side-exchanging fractional powers. Table 1 in Section~\ref{sec:SP-SEcompare} lists the occurrences  of essential fractional powers of $t_C$ in $\Mod(S_{g+1})$ and their exponents for genera $g$ in the range $20 \leq g + 1 \leq 30$. 

\section{Side-preserving fractional powers}
\label{sec:SP}

In this section, we will derive necessary and sufficient conditions for the existence of a side-preserving fractional power of exponent $\fraction{\ell}{n}$ and some additional applications. In fact, the main result of~\cite{MK1} can be extended to describe the side-preserving fractional powers of $t_C$. Adapting the main definition from that paper, we define an \textit{SP data set of exponent $\fraction{\ell}{n}$}.

\begin{definition}
\label{def:spdataset}
An \textit{SP data set of exponent $\fraction{\ell}{n}$} is a tuple of the form $((\ell,n),g_0,(a,b); (k_1,n_1),\ldots,(k_m,n_m))$ where:
\begin{enumerate}
\item[(i)] $\ell$, $n$, $g_0$, and the $n_i$ are integers such that $n>1$, $g_0\geq 0$, each $n_i>1$, and each $n_i$ divides $n$,
\item[(ii)] $a$ and $b$ are residues modulo $n$ with $\gcd(a,n) = \gcd(b,n) = 1$, and each $k_i$ is a residue modulo $n_i$ with $\gcd(k_i,n_i)=1$,
\item[(iii)] $a + b \equiv \ell ab\bmod n$, and
\item[(iv)] $a + b + \displaystyle\sum_{i=1}^m \frac{n}{n_i}k_i \equiv 0\bmod n$.
\end{enumerate}
\end{definition}

\noindent The integer $g$ defined by
\begin{equation}
\label{sp-eq}
g=g_0n + \frac{1}{2} \sum_{i=1}^m \frac{n}{n_i}(n_i-1)
\end{equation}
\noindent is called the \textit{genus} of the data set.

Two SP data sets are considered to be the same if they differ by interchanging $a$ and $b$ or by reordering of the pairs $(k_1,n_1),\ldots,(k_m,n_m)$. If $m = 0$ in Definition~\ref{def:spdataset}, then condition (iv) would give $b \equiv -a \bmod n$, which when substituted in (iii), would imply that $\ell \equiv 0 \bmod n$. So we may assume that $m \geq 1$.

\begin{remark}
\label{rem:nlowerbound}
When $g_0 = 0$ and $m=1$, Equation~\ref{sp-eq} takes the form 
\[2g = n(1-\frac{1}{n_1})\ .\]
Since $n_1 \leq n$, we have that
\[2g \leq n(1 - \frac{1}{n}) = n - 1\ ,\] that is, $n \geq 2g+1$.
\end{remark}

The proof of Theorem~1.1 in~\cite{MK1} adapts easily, as we will explain, to give the following:

\begin{theorem}
\label{thm:SP-main}
For a given $n>1$, $1 \leq \ell < n$, and $g \geq 1$, SP data sets of genus $g$ and exponent $\fraction{\ell}{n}$ correspond to the
conjugacy classes in $\Mod(S_{g+1})$ of the side-preserving fractional powers of $t_C$ of exponent $\fraction{\ell}{n}$.  Consequently, $t_C$ has a
side-preserving fractional power of exponent $\fraction{\ell}{n}$ if and only if there exists an SP data set of genus $g$ and exponent $\fraction{\ell}{n}$.
\end{theorem}

\begin{proof}
The first part of the proof of Theorem~1.1 of ~\cite{MK1} analyzed the case when $h$ is a side-preserving fractional power of $t_C$, obtaining a
$C_n$-action on a closed surface $S_g$ with two fixed points $P$ and $Q$ (and possibly other  with nontrivial stabilizers). The analysis here
proceeds in exactly the same way, to the point when the rotation angles at $P$ and $Q$ are analyzed. The condition that $a+b \equiv ab\bmod n$ was shown to
be equivalent to the fact that the rotation angles at the two ends of the annulus $N$ in $S_{g+1}$ differ by $2\pi/n$, so that on $N$, $h$ must be have
left-handed twisting of $2\pi/n$. An analogous argument shows that the condition $a+b \equiv \ell ab\bmod n$ is equivalent to $h$ having left-handed twisting
through $2\pi \ell/n$, so that $h^n=t_C^\ell$. Thus an SP data set of exponent $\fraction{\ell}{n}$ produces a fractional power of exponent
$\fraction{\ell}{n}$. The next part of the proof of Theorem~1.1 showed that side-exchanging roots of $t_C$ do not exist, which is irrelevant to us since we
are assuming that $h$ is side-preserving. Finally, the converse is a matter of reversing the argument. The arguments for proving that a $h'$ conjugate to $h$ would also yield the same SP data set
and the converse are analogous to the arguments in the proof of Theorem~1.1. 
\end{proof}

Theorem~\ref{thm:SP-main} allows us to perceive the conjugacy classes of side-preserving fractional powers of $t_C$ on $S_{g+1}$ simply as SP data sets. 
So for a given $g \geq 1$, we can explicitly compute the various possible exponents of side-preserving fractional powers of in $\Mod(S_{g+1})$ using the algebraic conditions on SP data sets. In the following proposition, we will show that a fractional power of of $t_C$ of exponent $\fraction{\ell}{n}$ is the $\ell^{th}$ of a $n^{th}$ root when $\ell$ and $n$ are relatively prime.

\begin{proposition}
Suppose that $h$ is a side-preserving fractional power of $t_C$ of exponent $\fraction{\ell}{n}$ with $\gcd(\ell,n)=1$. Then $h = {(h')}^\ell$ for some root $h'$ of~$t_C$ of degree $n$.
\label{coro:SP-powerofroot}
\end{proposition}

\begin{proof}
Describe the conjugacy class of $h$ by an SP data set \\$D = ((\ell,n),g_0,(a,b); (k_1,n_1),\ldots,(k_m,n_m))$, with $a+b \equiv \ell ab\bmod n$. Consider the tuple $D'$ obtained by replacing $\ell$ with 1, and 
multiplying the values $a$, $b$, and $k_1,\ldots$, $k_m$ by $\ell$. Since $\gcd (\ell,n) = 1$, $D'$ satisfies condition (ii) of an SP data set. Moreover, 
the fact that $\ell a + \ell b \equiv \ell a \, \ell b \bmod n$ would imply that $D'$ also satisfies condition (iii) of an SP data set. In other words, $D'$ represents 
a side-preserving fractional power $h'$ of degree $\fraction{1}{n}$, that is, $h'$ is a root of $t_C$ of degree $n$.

Recall the proof of Theorem 1.1 of~\cite{MK1}. The numerical data corresponding to $h$ described an orbifold $\O = S_g/C_n$ and an orbifold covering
$S_g \to \O$ corresponding to the kernel of a homomorphism $\pi_1^{orb}(\O)\to C_n$. The restriction $h_0$ of $h$ to a subsurface of $S_g$ was extended to an
annulus $N$, giving $h$ on $S_{g+1}$. The $h'$ above is obtained from the same orbifold $\O$ as $h$ is, but using a restriction of $h_0'$ of a different
covering transformation of $S_g$. Both have order $n$, so $h_0$ equals some power of $h_0'$. On $N$, $h^n = t_C^{\ell} = (h')^{\ell n}$, so that power is the
$\ell^{th}$ power (this can also be checked directly by examining the rotation angles of $h$ and $h'$ about the points $P$ and $Q$ in $F$).
\end{proof}

We now give an example to illustrate Proposition~\ref{coro:SP-powerofroot}.

\begin{example}
The SP data set $D = ((2, 9), 0, (1, 1); (7, 9))$, which represents the conjugacy class of a side-preserving fractional power of $t_C$ of exponent $\fraction{2}{9}$ in $\Mod(S_5)$, 
is the $2^{nd}$ power of a $9^{th}$ root of $t_C$, whose conjugacy class is given by the data set $D' = ((1, 9), 0, (2, 2); (5, 9))$. This is evident by multiplying $a = 1$, 
$b = 1$, and $c = 7$ of $D$ by 2 modulo 9, and then replacing $ \ell = 2$ with 1, to obtain $D'$.
\end{example}

An immediate application of Theorem~\ref{thm:SP-main} is the following corollary, where we derive and upper bound for $n$ when $\gcd(\ell,n) = 1$.

\begin{corollary}
\label{coro:SP-relupperbound}
Suppose that $h$ is a side-preserving fractional power of $t_C$ of degree $\fraction{\ell}{n}$. Then
\begin{itemize}
\item[(a)] $n$ is odd if $\ell$ is odd.
\item[(b)] $n \leq 2g+1$ if $\gcd(\ell,n) = 1$.
\end{itemize}
\end{corollary}

\begin{proof}
For a data set describing $h$, we have $a+b \equiv \ell ab\bmod n$. If $n$ is even, then $\ell$ must be even since $a$ and $b$ are relatively prime to~$n$.
This proves part $(a)$. For (b), suppose for contradiction that $n>2g+1$. From Equation~\ref{sp-eq}, we have that
\[1 > \frac{2g+1}{n} = \frac{1}{n} + 2g_0 + \sum_{i=1}^m (1 - \frac{1}{n_i})\ .\]

This would imply that $g_0=0$ and $m=1$, and consequently $n_1 < n$. Putting $d = n/n_1$, condition (iv) of Definition~\ref{def:spdataset} gives $a+b \equiv 0\bmod d$. 
Since $\gcd(\ell,n) = 1$ and $d$ divides $n$, this contradicts condition (iii) of Definition~\ref{def:spdataset}.
\end{proof}

\begin{remark}
Interestingly, the largest possible $\ell$ for which there exits a side-preserving fractional power of exponent $\fraction{\ell}{2g+1}$ is $2g$.
In fact, the SP data sets $D_1 = ((2g,2g+1), 0, (1, g); (g, 2g+1))$ and $D_2 = ((2g,2g+1), 0, (2g-1, 2g-1); (4, 2g+1))$ represent conjugacy classes of side-preserving fractional powers of $t_C$ of 
exponent $\fraction{2g}{2g+1}$ in $\Mod(S_{g+1})$.
\end{remark}

In the following corollary, we will derive an upper bound and a lower bound for $n$.

\begin{corollary}
\label{coro:SP-absupperbound}
Suppose that $h$ is a side-preserving fractional power of $t_C$ of exponent $\fraction{\ell}{n}$ whose conjugacy class is given by the SP data set
$D = ((\ell,n),g_0,(a,b);(c_1,n_1),\dots,(c_m,n_m))$. Then \[\frac{2g+m}{2g_0+m} \leq n \leq \frac{4g}{4g_0+m}\ .\]
\end{corollary}

\begin{proof}
To show that $n \geq \frac{2g+m}{2g_0+m}$, we use Equation~\ref{sp-eq} from the definition of an SP data set. On rewriting the equation, we get
\begin{equation}
\label{eqn:SP-simplified}
\frac{2g}{n} = 2g_0 + \sum_{i=1}^m (1 - \frac{1}{n_i})
\end{equation}
Since each $x_i \leq n$, we have that
\[ \frac{2g}{n} \leq 2g_0 + m - \frac{m}{n} \ ,\] from which we obtain the required inequality.

For the latter inequality, we use the fact that $n_i \geq  2$ in Equation~\ref{eqn:SP-simplified} above to obtain
 \[ \frac{2g}{n} \geq 2g_0 + \frac{m}{2}\ ,\] which upon simplification gives the inequality.
\end{proof}

The following corollary follows almost immediately from Corollary~\ref{coro:SP-absupperbound}.

\begin{corollary}
\label{coro:SP-mlbound}
Suppose that $h$ is a side-preserving fractional power of $t_C$ of exponent $\fraction{\ell}{n}$ given by the SP data set $D = ((\ell,n),g_0,(a,b);(c_1,n_1),\dots,(c_m,n_m))$. Then
\begin{itemize}
\item[(a)] $n \leq 4g$,
\item[(b)] $n < g$, whenever $g_0 \geq 1$, and
\item[(c)] $m = 1$, whenever $n > 2g$.
\end{itemize}	
\end{corollary}

\begin{remark}
\label{rem:ub-SP}
The upper bound for $n$ in Corollary~\ref{coro:SP-mlbound} is realizable since there always exist side-preserving fractional powers of exponent $\fraction{\ell}{4g}$.
For example, the data sets $D_1 = ((2g,4g), 0, (1,2g-1); (1, 2))$ and $D_2 = ((2g,4g), 0, (2g+1,4g-1); (1, 2))$ represent conjugacy classes of side-preserving fractional powers of $t_C$ of
exponent $\fraction{2g}{4g}$ in $\Mod(S_{g+1})$.
\end{remark}

Geometrically, data sets with $g_0 = 0$ and $m = 1$ represent conjugacy classes of essential fractional powers. From Remark~\ref{rem:nlowerbound},  we know that $n \geq 2g+1$, whenever $g_0 = 0$ and $m = 1$. We now list all such SP data sets that represent conjugacy classes of side-preserving fractional powers of exponent~\fraction{\ell}{n} in $\Mod(S_5)$.

\noindent Exponent $\fraction{1}{9}$:
\begin{itemize}
\item[(i)] $D_1 = ((1, 9), 0, (2, 2); (5, 9))$.
\item[(ii)]$D_2 = ((1, 9), 0, (5, 8); (5, 9))$,
\end{itemize}

\noindent Exponent $\fraction{2}{9}$:
\begin{itemize}
\item[(i)] $D_1 = ((2, 9), 0, (1, 1); (7, 9))$.
\item[(ii)] $D_2 = ((2, 9), 0, (4, 7); (7, 9))$.
\end{itemize}

\noindent Exponent $\fraction{4}{9}$:
\begin{itemize}
\item [(i)] $D_1 = ((4, 9), 0, (2, 8); (8, 9))$.
\item [(ii)] $D_2 = ((4, 9), 0, (5, 5); (8, 9))$.
\end{itemize}

\noindent Exponent $\fraction{5}{9}$:
\begin{itemize}
\item [(i)] $D_1 = ((5, 9), 0, (1, 7); (1, 9))$.
\item [(ii)] $D_2 = ((5, 9), 0, (4, 4); (1, 9))$.
\end{itemize}

\noindent Exponent $\fraction{7}{9}$:
\begin{itemize}
\item [(i)] $D_1 = ((7, 9), 0, (2, 5); (2, 9))$.
\item [(ii)] $D_2 = ((7, 9), 0, (8, 8); (2, 9))$.
\end{itemize}

\noindent Exponent $\fraction{8}{9}$:
\begin{itemize}
\item [(i)] $D_1 = ((8, 9), 0, (1, 4); (4, 9))$.
\item [(ii)] $D_2 = ((8, 9), 0, (7, 7); (4, 9))$.
\end{itemize}

\noindent Exponent $\fraction{2}{10}$:
\begin{itemize}
\item [(i)] $D_1 = ((2, 10), 0, (1, 1); (4, 5))$.
\item [(ii)] $D_2 = ((2, 10), 0, (7, 9); (2, 5))$.
\end{itemize}

\noindent Exponent $\fraction{4}{10}$:
\begin{itemize}
\item [(i)] $D_1 = ((4, 10), 0, (1, 7); (1, 5))$.
\item [(ii)] $D_2 = ((4, 10), 0, (3, 3); (2, 5))$.
\end{itemize}

\noindent Exponent $\fraction{6}{10}$:
\begin{itemize}
\item [(i)] $D_1 = ((6, 10), 0, (3, 9); (4, 5))$.
\item [(ii)] $D_2 = ((6, 10), 0, (7, 7); (3, 5))$.
\end{itemize}

\noindent Exponent $\fraction{8}{10}$:
\begin{itemize}
\item [(i)] $D_1 = ((8, 10), 0, (1, 3); (3, 5))$.
\item [(ii)] $D_2 = ((8, 10), 0, (9, 9); (1, 5))$.
\end{itemize}

\noindent Exponent $\fraction{4}{12}$:
\begin{itemize}
\item[(i)] $D_1 = ((4, 12), 0, (5, 11); (2, 3))$.
\end{itemize}

\noindent Exponent $\fraction{8}{12}$:
\begin{itemize}
\item [(i)] $D_1 = ((8, 12), 0, (1, 7); (1, 3))$.
\end{itemize}

\noindent Exponent $\fraction{8}{16}$:
\begin{itemize}
\item[(i)] $D_1 = ((8, 16), 0, (1, 7); (1, 2))$.
\item[(ii)] $D_2 = ((8, 16), 0, (3, 5); (1, 2))$.
\item[(iii)] $D_3 = ((8, 16), 0, (9, 15); (1, 2))$.
\item[(iv)] $D_4 = ((8, 16), 0, (11, 13); (1, 2))$.
\end{itemize}

\noindent In the above classification, it may be noted that the side-preserving fractional powers of exponent $\fraction{\ell}{9}$, for $\ell = 2, 4, 5, 7, \text{ and } 8$ are all powers of the 
fractional powers of exponent $\fraction{1}{9}$, that is, the ninth roots of $t_C$ on $S_5$. Moreover, the highest value that $\ell$ takes is $2g = 8$. 
These computations were made using the help of software~\cite{SPSOFTWARE} written in the GAP programming language.

\section{Side-exchanging fractional powers}
\label{sec:SE}
In this section, we shall derive equivalent conditions for the existence of side-exchanging fractional powers of $t_C$ on $S_{g+1}$. 
The geometric construction of side-exchanging fractional powers differs from that of side-preserving powers, as they 
are obtained from $C_{2n}$-action on $S_g$ with two distinguished fixed points that correspond to a unique cone point of order $n$ in the quotient orbifold.
Therefore, we need to analyze a slightly different kind of orbifold in this case, which motivates the following the following definition of an \textit{SE data set}.

\begin{definition}
A \textit{SE data set of exponent $\fraction{\ell}{2n}$} is a tuple $((\ell,2n),g_0,a; (k_1,n_1)\-,\ldots,(k_m,n_m))$ where:
\begin{enumerate}
\item[(i)] $\ell$, $n$, $g_0$, and the $n_i$ are integers such that $\ell \geq 2$, $n \geq 2$, $g_0 \geq 0$, each $n_i>1$, and each $n_i$ divides $2n$,
\item[(ii)] $a$ is a residue modulo $n$ with $\gcd(a,n) = 1$, and each $k_i$ is a residue modulo $n_i$ with $\gcd(k_i,n_i)=1$,
\item[(iii)] $\ell\,a\equiv2\bmod n$, and
\item[(iv)] $2a + \displaystyle\sum_{i=1}^m \frac{2n}{n_i}k_i \equiv 0 \bmod 2n$.
\end{enumerate}
\end{definition}
\noindent The integer $g$ defined by
\[g=n(2g_0-1) + \sum_{i=1}^m \frac{n}{n_i}(n_i-1)\]
is called the \textit{genus} of the data set.

Two SE data sets are considered to be same if they differ by the reordering of the pairs $(k_1,n_1),\ldots,(k_m,n_m)$.

\begin{remark}
\label{rem:SE-m}
As in case of SP data sets, it is apparent here too that if $m = 0$, then $\ell \equiv 0 \bmod n$. If $m = 1$ and $g_0 = 0$, then from Equation~\ref{se-eq}, we have
\[\frac{1-g}{n} = \frac{1}{n} + \frac{1}{x_1}\, \] which would imply that
\[-\frac{g}{n} = \frac{1}{x_1}\, \]which is impossible.
\end{remark}

We will now establish the main theorem in this section in which we will show that SE data sets correspond to conjugacy classes of side-exchanging fractional powers.

\begin{theorem} For a given $n\geq 1$, $1\leq \ell\leq n$, and $g\geq 0$, the SE data sets of genus $g$ and exponent $\fraction{\ell}{2n}$ correspond to the conjugacy classes in 
$\Mod(S_{g+1})$ of the side-exchanging fractional powers of $t_C$ of exponent $\fraction{\ell}{2n}$. Consequently, $t_C$ has a side-exchanging 
fractional power of exponent $\fraction{\ell}{2n}$ if and only if there exists a data set of genus $g$ and exponent $\fraction{\ell}{2n}$.
\label{thm:SE-main}
\end{theorem}

\begin{proof}
Suppose that $h$ is a side-exchanging fractional power of exponent $\fraction{\ell}{2n}$. As in the first part of the proof of Theorem~1.1, we use 
$h$ to obtain a $C_{2n}$-action on the closed surface $S_g$. Since $h$ exchanges 
the sides of $C$, $h^2$ preserves the sides of $C$ and hence the centers of the coned disks, $P$ and $Q$. Since the actions at $P$ and $Q$ 
are conjugate by a homeomorphism $h'$ that generates $C_{2n}$, rotation angles at $P$ and $Q$ must be equal to $2\pi k /n$ for some integer $k$ as indicated in Figure 1.

\begin{figure}[h]
\label{twisting}
\labellist
\small
\pinlabel $h^2(A)$ [B] at -18 25
\pinlabel $h^2(B)$ [B] at 300 25
\pinlabel $h$ [B] at 400 190
\pinlabel $h$ [B] at 370 650
\pinlabel $h(B)$ [B] at 390 512
\pinlabel $h(A)$ [B] at 616 350
\pinlabel $h^2$ [B] at 150 420
\pinlabel $A$ [B] at 25 790
\pinlabel $B$ [B] at 275 790
\pinlabel $P$ [B] at 35 105
\pinlabel $Q$ [B] at 239 105
\endlabellist
\centering
\includegraphics[width = 45 ex]{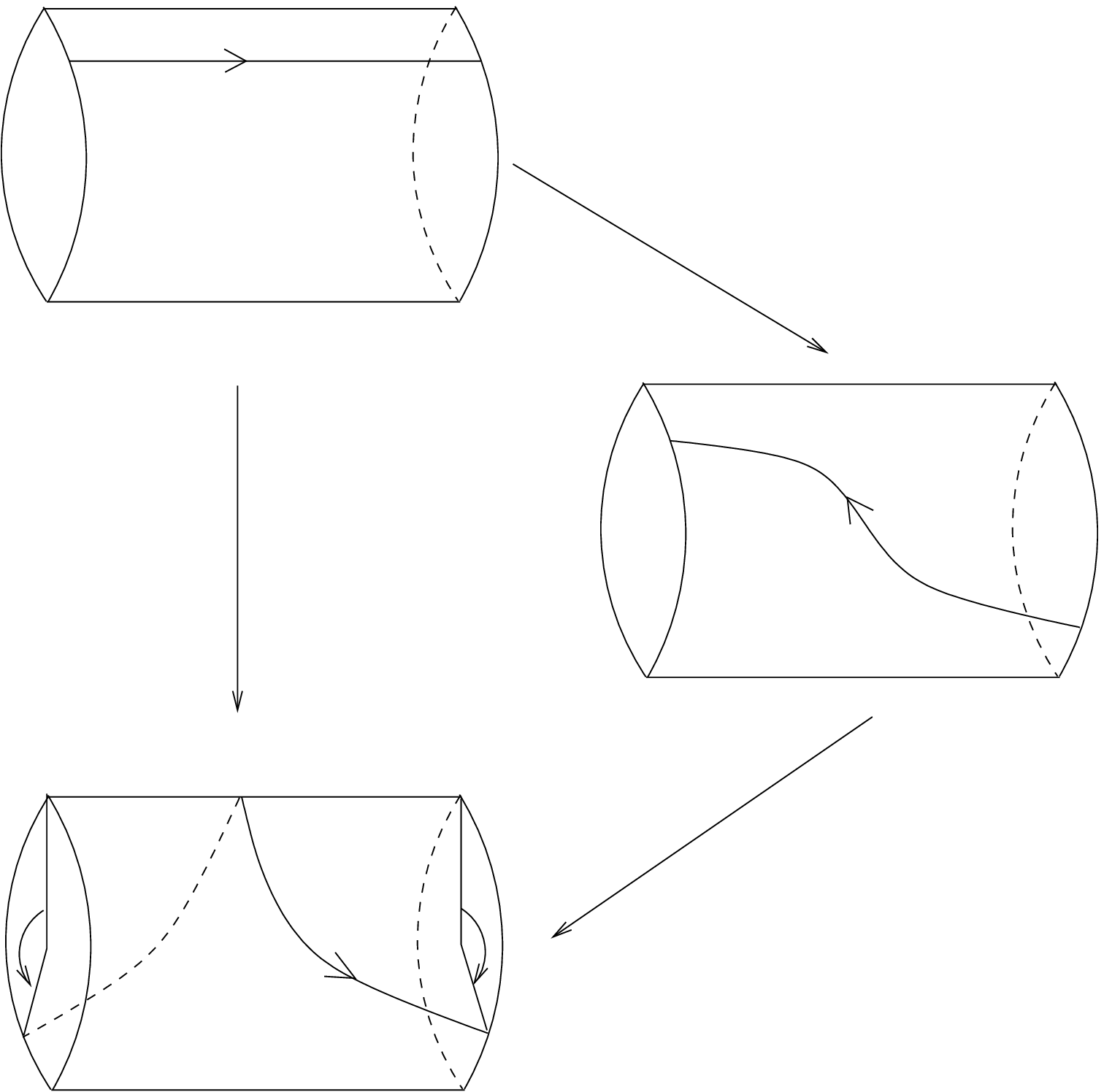}
\caption{The figure indicates the local effect of $h$ and $h^2$ on disk neighborhoods of $P$ and $Q$. The rotation angle of $h^2$ at $P$ and $Q$ is $2\pi {k}/{n}$.}
\end{figure}

Let $\O$ be the quotient orbifold for the $C_{2n}$ action on $S_g$. Denote the genus of $\O$ by $g_0$, and select standard generators $a_i, b_i, 1 \leq i \leq g_0$ of the fundamental group of the underlying surface. $\O$ has one distinguished cone point, $p$, of order $n$, which is the image of the distinguished fixed $P$ and $Q$ under the quotient map, and possibly $m$ other cone points $x_i$, $1\leq i\leq m$.

From orbifold covering space theory, the orbifold covering map $S_g \to \O$ corresponds to an exact sequence
\[ 1 \longrightarrow \pi_1(S_g) \longrightarrow \pi_1^{orb}(\O)\stackrel{\rho}{\longrightarrow} C_{2n} \longrightarrow 1\ .\]
Here, $C_{2n}$ is the group of covering transformations, generated by $t$, and $t^2$ generates the stabilizer at $P$. Let $\alpha$ be the generator of $\pi_1^{orb}\O$ going around $p$ and $\gamma_i$ be generators going around the $x_i$, selected so that
\begin{gather*}
\pi_1^{orb}(\O)=\langle \alpha, \gamma_1,\ldots, \gamma_m,
a_1,b_1,\ldots, a_{g_0}, b_{g_0}\;\vert\;\\
\alpha^n=\gamma_1^{n_1}=\cdots =\gamma_m^{n_m}=1,\;
\alpha\gamma_1 \cdots \gamma_m=\prod_{j=1}^{g_0}[a_j,b_j]\;\rangle
\end{gather*}
$\rho(\alpha)$ is determined by the covering transformation corresponding to the loop $\alpha$, which fixes $P$ and has turning angle $2\pi k/n$ at $P$ for some $k$ with $\gcd(k,n)=1$. So in $C_{2n}$, $\alpha$
represents ${\left(h^2\right)}^{a} = h^{2a}$, where $ka \equiv 1 \bmod n$. The rotation angle around $Q$ is also $2\pi k/n$. Since $h^{2n} = {\left(h^2\right)}^n = {t_{g+1}}^{\ell}$, the left-hand twisting angle of $h^2$ along the
tubular neighborhood $N$ of $C$ is $2\pi\ell/n$. This requires $2\pi k/n-(-2\pi k/n)=2\pi\ell/n$, giving $2k\equiv\ell\bmod n$. Multiplying by $a$ produces condition (iii) of a data set.

For $1\leq i\leq m$, the preimage of $x_i$ consists of $2n/n_i$ points cyclically permuted by $t$. Each of the points has stabilizer generated by $t^{2n/n_i}$. The rotation angle of $t^{2n/n_i}$ must be the same at all points of the orbit, since its action at one point is conjugate by a power of $t$ to its action at each other point. So the rotation angle at each point is of the form $2\pi k_i'/n_i$, where $\gcd(k_i',n_i)$, and as before, lifting $\gamma_i$ shows that $\displaystyle \rho(\gamma_i) = t^{(n/n_i)k_i}$ where $k_ik_i' \equiv 1\bmod n_i$. Since $C_{2n}$ is abelian, we have that $\displaystyle \rho(\prod_{j=1}^{g_0}[a_j,b_j])=1$, so
\[\displaystyle  1=\rho(\alpha\gamma_1\cdots \gamma_m)= t^{2a+(2n/n_1)k_1+ \cdots +(2n/n_m)k_m}\ ,\]
 giving condition (iv) of the data set.

The fact that the genus of the data set equals $g$ follows from the multiplicativity of the orbifold Euler characteristic for the orbifold covering $S_g \to \O$:
\begin{equation}
\label{se-eq}
\displaystyle  \frac{2-2g}{2n} = 2 - 2g_0 + \left(\frac{1}{n}-1\right) + \sum_{i=1}^m \left(\frac{1}{n_i}-1\right)
\end{equation}
Thus $h$ leads to a SE data set of exponent $\fraction{\ell}{2n}$. Finally, as in proof of Theorem~1.1 of~\cite{MK1}, the converse is a matter of reversing the argument. The arguments for establishing that an SE dataset would determine $h$ up to conjugacy and the converse are similar to the ones in the proof of Theorem~1.1. However, the part of that proof which pertains to showing that $\{P,Q\}$ is preserved by some conjugating homeomorphism is redundant in this case.
 \end{proof}
 
Theorem~\ref{thm:SE-main} allows to regard conjugacy class of a side-exchanging fractional power of exponent $\fraction{\ell}{2n}$ on $S_{g+1}$ simply as an SE data set. In the following proposition, we derive a condition under which a side-exchanging fractional power can be the power of another (side-exchanging) fractional power. 

\begin{proposition}
Let $h$ be a side-exchanging fractional power of $t_C$ of exponent $\fraction{\ell}{2n}$ such that $\ell$ is composite integer with $\gcd(\ell,n)=1$. Let $r$ be a divisor of $\ell$. Then $h = {(h')}^r$ for some side-exchanging fractional power $h'$ of~$t_C$ of exponent $\fraction{\ell'}{2n}$.
\label{coro:SE-powerofroot}
\end{proposition}

\begin{proof}
Describe the conjugacy class of $h$ by an SE data set~$D = ((\ell,2n),g_0,a; (k_1,n_1),\ldots,(k_m,n_m))$, with $\ell a \equiv 2 \bmod n$. Consider the tuple $D'$ obtained by replacing $\ell$ with $\ell' = \ell/r$, and 
multiplying the values $a$ and $k_1,\ldots$, $k_m$ by $r$. Since $\gcd (\ell,n) = 1$, $D'$ satisfies condition (ii) of an SE data set. Also, the fact that $(l/r)\,ar \equiv \ell a \equiv 2 \bmod n$ would imply that $D'$ also satisfies condition (iii) of an SE data set. In other words, $D'$ represents a side-exchanging fractional power $h'$ of exponent $\fraction{\ell'}{2n}$, where $\ell' = \ell/r$. 

As in the proof of Theorem~\ref{thm:SE-main}, numerical data corresponding to $h$ described an orbifold $\O = S_g/C_{2n}$ and an orbifold covering
$S_g \to \O$ corresponding to the kernel of a homomorphism $\pi_1^{orb}(\O)\to C_{2n}$. The restriction $h_0$ of $h$ to a subsurface of $S_g$ was extended to an
annulus $N$, giving $h$ on $S_{g+1}$. The $h'$ above is obtained from the same orbifold $\O$ as $h$ is, but using a restriction of $h_0'$ of a different
covering transformation of $S_g$. Both have order $2n$, so $h_0$ equals some power of $h_0'$. Therefore, on $N$,  we have that $h^{2n} = t_C^{\ell} = (h')^{\ell' r}$.
\end{proof}

The following is a concrete example that illustrates Proposition~\ref{coro:SE-powerofroot}. 

\begin{example}
The SE data set $D = ((6, 10), 0, 2; (3, 10), (3, 10))$ that represents the conjugacy class of a side-exchanging fractional power $h$ of $t_C$ of exponent $\fraction{6}{10}$ in $\Mod(S_5)$, 
is the $2^{nd}$ power of a side-exchanging fractional power $h'$ of $t_C$ of exponent $\fraction{3}{10}$,  whose conjugacy class is described by the data set $D' = ((3, 10), 0, 4; (1, 10), (1, 10))$. It is quite apparent that $D'$ can be obtained from $D$ by multiplying $a = 2$, $k_1 = 3 $, and $k_1 = 3$ of $D$ by $r  = 2$ (modulo 5), and then replacing $\ell = 6$ with $\ell' = \ell/r = 3$.
\end{example}

\begin{remark}
\label{rem:SE-dataset}
It is a well-known result of W. J. Harvey~\cite{WJH} and A. Wiman~\cite{W1} that the largest order of a cyclic action on a closed orientable surface of genus $g$ is $4g+2$. For this reason, $2n \leq 4g+2$, and from our earlier assumption, $\ell \leq 4g+1$. So it is interesting to note that for $g \geq 1$, there exists a SE fractional power of $t_C$ of degree $\fraction{4g+1}{4g+2}$ in $\Mod(S_{g+1})$ and its conjugacy class is represented by the SE data set $D = ((4g+1,4g+2), 0, 2g-1; (1, 2), (2g+5, 4g+2))$.
\end{remark}

\noindent In the following corollary, we shall derive a lower bound for $2n$.

\begin{corollary}
\label{coro:SE-ellandn}
Suppose that $h$ is a SE fractional power of $t_C$ of exponent $\fraction{\ell}{2n}$ given by the SE data set $D = ((\ell,2n),g_0,a;(k_1,n_1),\ldots,(k_m,n_m))$. Then
\begin{itemize}
\item[(a)] $n$ is odd if $\ell$ is odd, and
\item[(b)] $2n \geq \frac{2g+m}{2g_0+m-1}$
\end{itemize}
\end{corollary}

\begin{proof}
The proof of (a) follows directly from conditions (ii) and (iii) in the definition of an SE data set. For, if $n$ is even, then $\ell$ must be even since $a$ is relatively prime to $n$.
To show (b), we use Equation~\ref{se-eq} from the proof of Theorem~\ref{thm:SE-main}, which upon simplification gives
\begin{equation}
\label{eq:SE-simplified}
-\frac{g}{n} = 1 - 2g_0 + \sum_{1=i}^{m}(\frac{1}{x_i}-1)
\end{equation}
Since $x_i \leq 2n$, we have that
\[-\frac{2g}{2n} \leq 1 - 2g_0  + \frac{m}{2n} - m\ .\] Since we know by Remark~\ref{rem:SE-m} that if $m = 1$ then $g_0 > 0$, we can infer that
\[2n \geq \frac{2g+m}{2g_0+m-1}\ .\]
\end{proof}

\begin{remark}
From Remark~\ref{rem:SE-m}, we know that $m \geq 2$ whenever $g_0 = 0$. Moreover, when $g_0 = 0$, it follows then from Corollary~\ref{coro:SE-ellandn} that $2n \geq 2g + 2$. It is worth mentioning here that (for $g \geq 1$) there always exist a side-exchanging fractional power of $t_C$ exponent $\fraction{2}{2g+2}$ in $\Mod(S_{g+1})$ and its conjugacy class is given by the SE data set $D = ((2, 2g+2), 0, 1; (2g+1,2g+2),(2g+1,2g+2))$.
\end{remark}

When $g_0 = 0$ and $\ell = 2$, below are the SE data sets that represent conjugacy classes of side-exchanging essential fractional powers of $t_C$ in $\Mod(S_5)$. 

\noindent Exponent $\fraction{2}{10}$:
\begin{itemize}
\item [(i)] $D_1 = ((2, 10), 0, 1; (1, 10), (7, 10))$.
\item [(ii)] $D_2 = ((2, 10), 0, 1; (9, 10), (9, 10))$.
\end{itemize}

\noindent Exponent $\fraction{3}{10}$:
\begin{itemize}
\item[(i)] $D_1 = ((3, 10), 0, 4; (1, 10), (1, 10))$.
\item [(ii)] $D_2 = ((3, 10), 0, 4; (3, 10), (9, 10))$.
\end{itemize}

\noindent Exponent $\fraction{4}{10}$:
\begin{itemize}
\item[(i)] $D_1 = ((4, 10), 0, 3; (1, 10), (3, 10))$.
\item[(ii)] $D_2 = ((4, 10), 0, 3; (7, 10), (7, 10))$.
\end{itemize}

\noindent Exponent $\fraction{6}{10}$:
\begin{itemize}
\item[(i)] $D_1 = ((6, 10), 0, 2; (3, 10), (3, 10))$.
\item[(ii)] $D_2 = ((6, 10), 0, 2; (7, 10), (9, 10))$.
\end{itemize}

\noindent Exponent $\fraction{7}{10}$:
\begin{itemize}
\item[(i)] $D_1 = ((7, 10), 0, 1; (1, 10), (7, 10))$.
\item[(ii)] $D_2 = ((7, 10), 0, 1; (9, 10), (9, 10))$.
\end{itemize}

\noindent Exponent $\fraction{8}{10}$:
\begin{itemize}
\item[(i)] $D_1 = ((8, 10), 0, 4; (1, 10), (1, 10))$.
\item[(ii)] $D_2  = ((8,10), 0, 4; (3, 10), (9, 10))$.
\end{itemize}

\noindent Exponent $\fraction{9}{10}$:
\begin{itemize}
\item[(i)] $D_1 = ((9, 10), 0, 3; (1, 10), (3, 10))$.
\item[(ii)] $D_2 = ((9, 10), 0, 3; (7, 10), (7, 10))$.
\end{itemize}

\noindent Exponent $\fraction{2}{12}$:
\begin{itemize}
\item[(i)] $D_1 = ((2, 12), 0, 1; (1, 4), (7, 12))$.
\item[(ii)] $D_2 = ((2, 12), 0, 1; (3, 4), (1, 12))$.
\end{itemize}

\noindent Exponent $\fraction{4}{12}$:
\begin{itemize}
\item[(i)] $D_1 = ((4, 12), 0, 5; (1, 4), (11, 12))$.
\item[(ii)] $D_2 = ((4, 12), 0, 5; (3, 4), (5, 12))$.
\end{itemize}

\noindent Exponent $\fraction{8}{12}$:
\begin{itemize}
\item[(i)] $D_1 = ((8, 12), 0, 1; (1, 4), (7, 12))$.
\item[(ii)] $D_2 = ((8, 12), 0, 1; (3, 4), (1, 12))$.
\end{itemize}

\noindent Exponent $\fraction{10}{12}$:
\begin{itemize}
\item[(i)] $D_1 = ((10, 12), 0, 5; (1, 4), (11, 12))$.
\item[(ii)] $D_2 = ((10, 12), 0, 5; (3, 4), (5, 12))$.
\end{itemize}

\noindent Exponent $\fraction{2}{18}$:
\begin{itemize}
\item[(i)] $D_1 = ((2, 18), 0, 1; (1, 2), (7, 18))$.
\end{itemize}

\noindent Exponent $\fraction{4}{18}$:
\begin{itemize}
\item[(i)] $D_1 = ((4, 18), 0, 5; (1, 2), (17, 18))$.
\end{itemize}

\noindent Exponent $\fraction{5}{18}$:
\begin{itemize}
\item[(i)] $D_1 = ((5, 18), 0, 4; (1, 2), (1, 18))$.
\end{itemize}

\noindent Exponent $\fraction{7}{18}$:
\begin{itemize}
\item[(i)] $D_1 = ((7, 18), 0, 8; (1, 2), (11, 18))$.
\end{itemize}

\noindent Exponent $\fraction{13}{18}$:
\begin{itemize}
\item[(i)] $D_1 = ((8, 18), 0, 7; (1, 2), (13, 18))$.
\end{itemize}

\noindent Exponent $\fraction{10}{18}$:
\begin{itemize}
\item[(i)] $D_1 = ((10, 18), 0, 2; (1, 2), (5, 18))$.
\end{itemize}

\noindent Exponent $\fraction{11}{18}$:
\begin{itemize}
\item[(i)] $D_1 = ((11, 18), 0, 1; (1, 2), (7, 18))$.
\end{itemize}

\noindent Exponent $\fraction{13}{18}$:
\begin{itemize}
\item[(i)] $D_1 = ((13, 18), 0, 5; (1, 2), (17, 18))$.
\end{itemize}

\noindent Exponent $\fraction{14}{18}$:
\begin{itemize}
\item[(i)] $D_1 = ((14, 18), 0, 4; (1, 2), (1, 18))$.
\end{itemize}

\noindent Exponent $\fraction{11}{18}$:
\begin{itemize}
\item[(i)] $D_1 = ((16, 18), 0, 8; (1, 2), (11, 18))$.
\end{itemize}

\noindent Exponent $\fraction{17}{18}$:
\begin{itemize}
\item[(i)] $D_1 = ((17,18), 0, 7; (1, 2), (13, 18))$.
\end{itemize}

\section{The occurrence of side-exchanging and side-preserving fractional powers}
\label{sec:SP-SEcompare}

In this section, we shall make a general comparison between the occurrences of side-exchanging and side-preserving essential fractional powers of $t_C$ and their exponents in  $\Mod(S_{g+1})$. We will use following notation. 

\begin{notation}
We will denote the number of distinct exponents of side-exchanging and side-preserving essential fractional powers of $t_C$ in $\Mod(S_{g+1})$, respectively,  by $E_{SE}(g)$ and $E_{SP}(g)$. Also, we will denote total number of side-exchanging and side-preserving essential fractional powers of $t_C$ (up to conjugacy) in $\Mod(S_{g+1})$ , respectively,  by $N_{SE}(g)$ and $N_{SP}(g)$.
\end{notation}

For $20 \leq g + 1 \leq 30$, Table 1 gives $E_{SP}(g)$, $E_{SE}(g)$, $N_{SP}(g)$, and $N_{SE}(g)$.
\begin{table}
\small
\begin{center}
    \begin{tabular}{ | c | c | c | c | c | c |}
    \hline
    $g+1$ & $E_{SP}(g)$ & $E_{SE}(g)$ & $N_{SP}(g)$ & $N_{SE}(g)$ \\ \hline
     20 & $35$ & 102 & 236 & 322\\ \hline
     21 & $77$ & 102 & 1034 & 148\\ \hline
    22 & $75$ & 103 & 1284 & 283\\ \hline
    23 & $57$ & 188 & 468 & 906\\ \hline
    24 & $57$ & 99 & 1142 & 171\\ \hline
    25 & $111$ & 134 & 1498 & 491\\ \hline
    26 & $59$ & 154 & 628 & 625\\ \hline
    27& $83$ & 193 & 1610 & 349\\ \hline
    28 & $85$ & 146 & 1208 & 414\\ \hline
    29 & $89$ & 178 & 930 & 1009\\ \hline
    30 & $69$ & 178 & 1770 & 226\\ \hline
   \end{tabular}
\end{center}
\caption{This data illustrates that, in general, $E_{SP}(g)  < E_{SE}(g)$, while $N_{SP}(g) > N_{SE}(g)$. The data seems to indicate that though the side-exchanging fractional posers achieve more exponents, they are in general fewer in number when compared with side-preserving fractional powers }
\end{table}


\bibliographystyle{amsplain}
\bibliography{frac-root}
\end{document}